\documentclass[12pt]{amsart}

\usepackage{amssymb}

\usepackage[shortlabels]{enumitem}
\setlist[enumerate]{leftmargin=*,font=\upshape,align=parleft,label=(\alph*)}
\setlist[itemize]{leftmargin=*}

\usepackage[alphabetic, nobysame]{amsrefs}

\usepackage{hyperref}
\hypersetup{
	pdfstartview={XYZ null null 1.00}, 
	pdfpagemode=UseNone, 
	colorlinks,
	breaklinks, 
	linkcolor=blue,
	urlcolor=blue, 
	anchorcolor=blue,
	citecolor=blue
}

\usepackage[capitalise]{cleveref}

\usepackage{geometry}
\usepackage{xspace}

\usepackage{graphicx}
\theoremstyle{plain}

\newtheorem{lemma}[equation]{Lemma}

\crefname{prop}{Proposition}{Propositions}
\newtheorem{prop}[equation]{Proposition}

\newtheorem{theorem}[equation]{Theorem}

\crefname{obs}{Observation}{Observations}

\crefname{cor}{Corollary}{Corollaries}
\newtheorem{cor}[equation]{Corollary}

\theoremstyle{definition}

\crefname{defn}{Definition}{Definitions}
\newtheorem{defn}[equation]{Definition}

\crefname{example}{Example}{Examples}
\newtheorem{example}[equation]{Example}

\crefname{question}{Question}{Questions}
\newtheorem{question}[equation]{Question}

\crefname{problem}{Problem}{Problems}

\newtheorem*{ECS_problem}{Maximum $k$-edge colorable subgraph problem ($k$-ECS)}

\crefname{conj}{Conjecture}{Conjectures}

\theoremstyle{remark}

\crefname{remark}{Remark}{Remarks}
\newtheorem{remark}[equation]{Remark}

\crefname{claim+}{Claim}{Claims}
\newtheorem{claim+}[equation]{Claim}

\newcommand{\N}{\mathbb{N}}
\newcommand{\Z}{\mathbb{Z}}

\renewcommand{\phi}{\varphi}

\newcommand*{\defeq}{\mathrel{\vcenter{\baselineskip0.5ex \lineskiplimit0pt \hbox{\scriptsize.}\hbox{\scriptsize.}}}=}

\newcommand{\set}[1]{\left\{ #1 \right\}}

\newcommand{\p}{p}
\newcommand{\ep}{ep}
\newcommand{\epp}{ep^*}

\title{Characterization of saturated graphs related to pairs of disjoint matchings}

\author{Zhengda Mo}
\author{Sam Qunell}

\author{Anush Tserunyan}
\author{Jenna Zomback}

\thanks{This work is part of the research project ``Pairs of disjoint matchings'' within \href{https://math.illinois.edu/research/igl}{Illinois Geometry Lab} in Spring 2019 -- Spring 2020. The first and second authors participated as undergraduate scholars, the fourth author served as graduate student team leader, and the third author as faculty mentor. The third author was supported by the NSF Grant DMS-1855648.}

\newcommand{\skeleton}{skeleton \xspace}

\begin{document}

\maketitle
\begin{abstract}
    For a finite graph $G$, we study the maximum $2$-edge colorable subgraph problem and a related ratio $\frac{\mu(G)}{\nu(G)}$, where $\nu(G)$ is the matching number of $G$, and $\mu(G)$ is the size of the largest matching in any pair $(H,H')$ of disjoint matchings maximizing $|H| + |H'|$ (equivalently, forming a maximum $2$-edge colorable subgraph). Previously, it was shown that $\frac{4}{5} \le \frac{\mu(G)}{\nu(G)} \le 1$, and the class of graphs achieving $\frac{4}{5}$ was completely characterized. In this paper, we first show that graph decompositions into paths and even cycles provide a new way to study these parameters. We then use this technique to characterize the graphs achieving $\frac{\mu(G)}{\nu(G)} = 1$ among all graphs that can be covered by a certain choice of a maximum matching and $H$, $H'$ as above.
    

\end{abstract}

\section{Introduction}\label{sec:intro}

We investigate the ratio of two parameters of a finite graph $G$, $\frac{\mu(G)}{\nu(G)}$, where $\nu(G)$ is the size of a maximum matching in $G$, and $\mu(G)$ is the size of the largest matching in a pair of disjoint matchings whose union is as large as possible. More formally,
\begin{align*}
\lambda(G) 
\defeq& 
\max\set{|H|+|H'| : \text{$H$ and $H'$ are disjoint matchings in $G$}},
\\
\mu(G)
\defeq&
\max\set{|H|:\text{$H$ and $H'$ are disjoint } \text{matchings in $G$ with }|H|+|H'|=\lambda(G)}.
\end{align*}

To place the mentioned parameters $\lambda,\mu$ and the ratio $\frac{\mu}{\nu}$ in the context of what is commonly considered in graph theory, recall that for $k \in \N$, a proper edge-coloring of a graph $G$ with $k$ colors is exactly a partition of the edge-set into $k$ disjoint matchings. Thus, determining $\lambda$ is the $k=2$ case of the following well-known problem: 

\begin{ECS_problem}\label{problem:k-covering}
For a given $k \in \N$, what is the maximum number $\nu_k$ of edges of a given finite graph $G$ that can be covered with $k$ disjoint matchings?
\end{ECS_problem}

We note here that $\nu = \nu_1$ and $\lambda = \nu_2$, so the $k=1$ case of this problem is the maximum matching problem, while the $k=2$ case is what we are concerned with here. Before describing our result, we overview what is known about $k$-ECS concerning its computational complexity, numerical bounds, and structural aspects.

\subsection*{Computational complexity}

The simple case of $1$-ECS is completely solved as it admits a polynomial-time algorithm \cites{Edmonds:1965}. 
On the other hand, $2$-ECS is already $NP$-hard; in fact, the problem of determining whether $\lambda(G)$ is equal to the number of vertices of the graph is $NP$-complete even for cubic graphs. This follows immediately from Holyer's work \cite{Holyer} by the following argument. Holyer proves that determining whether a graph is $3$-edge colorable is $NP$-complete even for cubic graphs. Letting $G \defeq (V,E)$ be a cubic graph, any $3$-edge-coloring gives $3$ disjoint perfect matchings, in particular, $\lambda=|V|$. Conversely, if $\lambda=|V|$, then for any pair $(H,H')$ of disjoint matchings with $|H|+|H'|=\lambda$, $H$ and $H'$ are both perfect matchings, and the remaining edges $H''$ in $E$ form a third perfect matching, so $H$, $H'$, $H''$ is a 3-edge-coloring of $G$. Thus, a cubic graph is $3$-edge colorable if and only if $\lambda=|V|$.

To give an intuitive idea why finding $\lambda(G)$ is hard, we illustrate how the greedy algorithm fails to find it. The algorithm would firstly take a maximum matching of the whole graph and then take a maximum matching of the remaining edges. However, this may not cover the maximum number of edges $\lambda(G)$: see \cref{fig:spanner} for an example where $\lambda(G) = 4+4 = 8$, while the greedy algorithm yields $5+2 = 7$. The ratio $\frac{\mu(G)}{\nu(G)}$ that we study in the present paper measures the failure of optimality of the greedy algorithm for $2$-ECS; in particular, this ratio is $\frac{4}{5}$ for the graph in \cref{fig:spanner}.

\begin{figure}[h]
    \centering
    \includegraphics[width=0.5\textwidth]{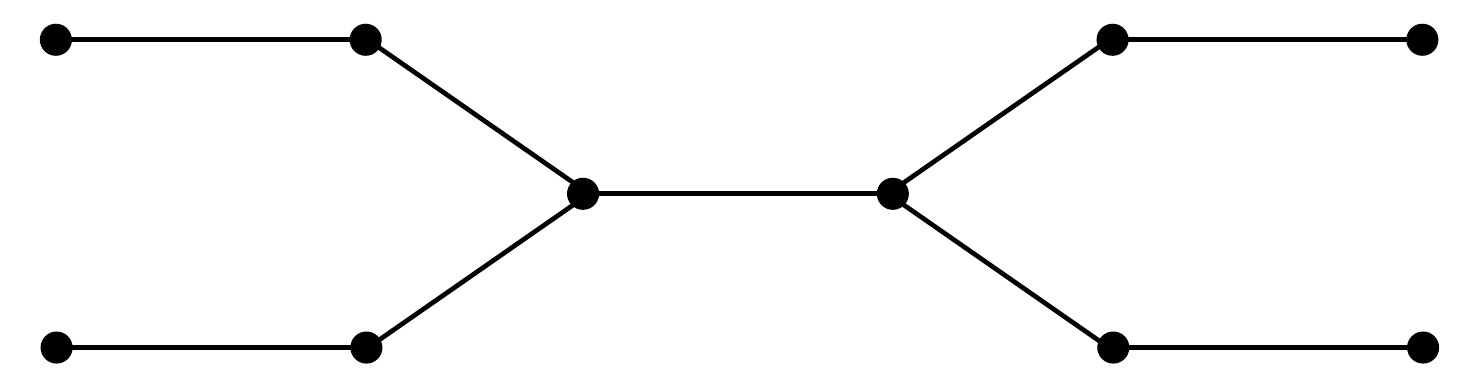}
    \caption{The spanner graph}
    \label{fig:spanner}
\end{figure}

In the positive direction for computing $\lambda(G)$, it is shown in \cite{Agrawal+} that finding $\lambda(G)$ is fixed-parameter tractable.

More generally for all $k \ge 2$, it is known that $k$-ECS is $NP$-hard \cite{Feige+} and APX-hard \cite{Kosowski:2009}. Related complexity bounds are shown in \cite{A-M:2019}.

\medskip

Although finding $\lambda$ is $NP$-hard, the complexity of finding $\mu(G)$ is still unknown.

\begin{question}
What is the computational complexity of finding $\mu(G)$? In particular, is it also $NP$-hard?
\end{question}

\begin{remark}
One may wonder whether there is an analogue for the Berge--Tutte formula \cites{Berge,Tutte} \cite{Lovasz-Plummer}[Theorem 3.1.14], which relates $\nu(G)$ to the number of odd connected components in subgraphs of $G$, for $\lambda$. However, that formula provides a witness to $\nu(G) \le k$ for a fixed $k$, hence implying that finding $\nu(G)$ is both in $NP$ and co-$NP$. Thus, since we already know that finding $\lambda(G)$ ($2$-ECS) is $NP$-hard, it is unlikely that there is an analogue of the Berge--Tutte formula for $\lambda(G)$ in general. However, such a formula may exist for subclasses of graphs.
\end{remark}

\subsection*{Numerical bounds}

In \cite{M-M-Ts:2008}, V. Mkrtchyan, V. Musoyan, and A. Tserunyan proved that
\[ 
\frac{4}{5} \leq \frac{\mu(G)}{\nu(G)} \leq 1 
\] 
for any graph $G$, and more recently, in \cite{IGL:2019}, it was shown that for any rational $m,n \in \N$ with $\frac{4}{5} \leq \frac{m}{n} \leq 1$, there is a connected graph $G$ with $\mu(G)=m$ and $\nu(G)=n$ (an explicit construction is provided).


As shown in \cite{A-M-P-V:2014}, this lower bound can be improved for the case of cubic graphs. In this case, we have \[ 
\frac{8}{9} \leq \frac{\mu(G)}{\nu(G)} \leq 1 
\] for a cubic graph $G$. 

A related line of research studies the general $\nu_k$, which is the most edges that can be colored by $k$ disjoint matchings (i.e. $\nu_2(G)=\lambda(G)$). In this direction, \cite{M-P-V:2010} shows that in cubic graphs,
\begin{align*}
    \lambda(G) &\ge \frac{4}{5} | V(G) | \text{ and}
    \\
    \lambda(G) &\le \frac{|V(G)| + 2\nu_3(G)}{4}.
\end{align*}

\noindent Lastly, \cite{K-M:2019} establishes the following family of inequalities for a bipartite graph $G$ and any $0\leq i \leq k$:
\[\nu_k\geq \frac{\nu_{k-i}+\nu_{k+i}}{2}.\]

\subsection*{Structural aspects and our result}

Besides the relevance and naturalness of $k$-ECS problem, our motivation for studying the parameters $\lambda$ and $\mu$, and the ratio $\frac{\mu}{\nu}$ is that they reveal interesting structural properties of the graph. Furthermore, we observe in \cref{sec:pec} below that the parameters $\lambda$, $\mu$, and $\nu$ are tied to decompositions of graphs into paths and even cycles, which provides a dual angle of interest for studying these parameters. 

It was shown in \cite{Tserunyan:2009} that the graphs minimizing the ratio $\frac{\mu}{\nu}$ (i.e. $\frac{\mu(G)}{\nu(G)} = \frac{4}{5}$) admit rigid structure; in fact, this paper provides a complete structural characterization of these graphs. \cref{fig:spanner} depicts the \textit{spanner} --- the unique minimal graph that achieves the ratio $\frac{4}{5}$. 

As for the graphs achieving the upper bound $1$ for the ratio $\frac{\mu}{\nu}$, the question is still open:

\begin{question}\label{q:ratio=1}
Is there a structural characterization of all graphs $G$ with $\frac{\mu(G)}{\nu(G)} = 1$?
\end{question}

This question has been considered previously. For example, Mkrtchyan in \cite{Mkrtchyan} shows that trees whose every edge is in a maximum matching satisfy $\frac{\mu(G)}{\nu(G)}=1$. A more general sufficient condition for achieving the upper bound $1$ is provided in \cite{IGL:2019}, where is it proven that any graph $G$ with $\frac{\mu(G)}{\nu(G)}<1$ must contain one of a special family of graphs as a subgraph. 

In the present paper, we answer \cref{q:ratio=1} for a core (with respect to the involved parameters) subclass of graphs, providing a complete structural characterization. To describe this class, we define a choice of \emph{maximally intersecting} matchings $M,H$, and $H'$ in \cref{maximal}, and we call a graph \textit{saturated} if its edges are covered by some maximally intersecting matchings $M$, $H$, and $H'$, where $M$ is a maximum matching and $H,H'$ are disjoint matchings with $|H|+|H'| = \lambda(G)$ and $|H| = \mu(G)$. These maximally intersecting matchings feature the core structure of graphs with $\frac{\mu(G)}{\nu(G)} < 1$, as the edges not included in these matchings can appear in limited configurations and do not contribute to the parameters $\mu(G)$ or $\nu(G)$. Among saturated graphs, we precisely characterize those with $\frac{\mu(G)}{\nu(G)}<1$ by equating this class with that of \textit{skeletons}. Referring the reader to \cref{defn:skeleton} for the precise definition, we roughly define a \textit{$k$-skeleton} here as a disjoint union of $k$-many odd leaf-to-leaf paths with some additional special odd paths intertwining them. Our main result is:

\begin{theorem}\label{label:char_via_skeletons}
    If a finite, connected graph $G$ is saturated and satisfies $\frac{\mu(G)}{\nu(G)}<1$, then it is a $k$-skeleton with $k=\nu(G)-\mu(G)$. Furthermore, if $G$ is a $k$-skeleton, then $G$ is saturated, and $\nu(G)-\mu(G)=k$. In particular, $\frac{\mu(G)}{\nu(G)}<1$ for all $k$-skeletons.
\end{theorem}

The precise statement of this result is given in \cref{theorem:skelehask,theorem:khasskele}.

In the proof of \cref{label:char_via_skeletons}, we use an alternative characterization of the parameters $\nu$, $\lambda$, $\mu$, which is interesting in its own right. This characterization replaces the maximization on matchings used in our definitions to minimization on decompositions into paths and even cycles, which we call \textit{PEC decompositions} (see \cref{sec:pec}).

As for unsaturated graphs, i.e. those with edges not covered with the mentioned three matchings, we provide in \cref{sec:extra-edges} a couple of structural restrictions on those extra edges and state a question. However, \cref{q:ratio=1} remains open in general.

\subsubsection*{Organization} 
\cref{sec:prelims} provides the preliminaries and sets up notation. PEC decompositions are discussed in \cref{sec:pec}. In \cref{sec:tools}, we establish the basic theory of the maximally intersecting matchings studied in this paper and the alternating paths that they form. Our main results, \cref{theorem:skelehask} and \cref{theorem:khasskele}, are proved in \cref{sec:ratio<1} using the techniques from the previous sections. We conclude in \cref{sec:extra-edges} with a discussion of those graphs that are not covered by these matchings.
 
\subsubsection*{Acknowledgements}
We thank Vahan Mkrtchyan for his invaluable insight into this direction of research and for pointing out the relevant literature and context. We also thank Alexandr Kostochka for helpful suggestions and advice.


\section{Definitions \& notation}\label{sec:prelims}

Throughout, we denote by $G=(V,E)$ a finite graph with vertex set $V$ and edge set $E$.
A \textbf{matching} in $G$ is a set of edges such that no two are adjacent. In particular, we will consider the following specific types of matchings.

\begin{itemize}
\item A \textbf{maximum matching} is a matching $ M $ of maximum size, i.e., for all matchings $ M' $, $ |M'| \leq |M| $. We let $\boldsymbol{\nu(G)}$ denote the size of a maximum matching.
\item A \textbf{perfect matching} is a matching $M$ such that for every vertex $v$ in $G$, there exists an $e\in M$ such that $v$ is incident to $e$.
\item A \textbf{near perfect matching} is a matching $M$ such that for every vertex $v$ in $G$ except for one, there exists an $e\in M$ such that $v$ is incident to $e$.
\end{itemize}

In addition to sizes of maximum matchings, we will also pay special attention to the size of the union of two disjoint matchings whose union is of maximum size. More formally:

\begin{itemize}
\item $\boldsymbol{\lambda(G)} := \max\{|H|+|H'|:(H,H') \text{ are disjoint matchings in } G\}$;

\item $\boldsymbol{\Lambda(G)}$ is the set of pairs $ (H, H') $ of disjoint matchings satisfying $|H| + |H'| = \lambda(G)$;

\item $\boldsymbol{\mu(G)} := \max\{|H| : \exists H' \text{ such that } (H, H') \in \Lambda(G)\}$.
\end{itemize}

$\nu$, $\lambda$, and $\mu$ are the main concern of this paper. We present an example with the spanner graph to illustrate our key parameters. The left spanner of \cref{fig:spanners} has the unique maximum matching, $M$, highlighted. It has size 5, and so $\nu(G)=5$ for the spanner. The largest matching disjoint from $M$ has size 2, for combined size of 7. However, if we choose a separate set of disjoint matchings, we can achieve $\lambda(G)=8$. On the right of the same figure, our maximal disjoint matchings $H$ and $H'$ are highlighted. They each have size $4$, so $\mu(G)=4$ for the spanner. Most graphs with $\frac{\mu(G)}{\nu(G)}<1$ share key structural properties with the spanner, as we will make clear in Section 5.
\begin{figure}[h]
    \centering
    \includegraphics[width=0.75\textwidth]{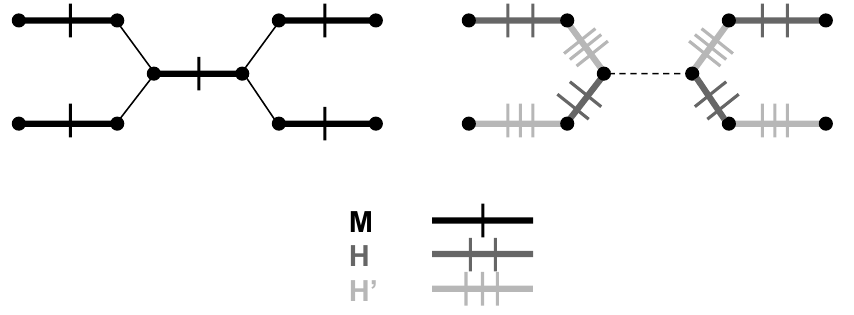}
     \caption{In the spanner, $\nu(G)=5$ but $\mu(G)=4$.}
    \label{fig:spanners}
\end{figure}

\begin{itemize}
\item $\boldsymbol{\Lambda_\mu(G)}$ is the subset of $\Lambda(G)$ of pairs $(H,H')$ with $|H| = \mu(G)$.
\item $\boldsymbol{\mu'(G)} : = \lambda(G) - \mu(G)$.

\item By a \textbf{chain} in $G$, we mean a sequence $(v_0, e_1, v_1, e_2, \dots, e_\ell, v_\ell)$, $\ell \in \N$, where the $v_i$ are vertices and the $e_j$ are edges of $G$ such $e_i = (v_{i-1},v_i)$ for all $i \le \ell$; we also assume that our chains are simple, i.e. all edges are distinct. We refer to $\ell$ as the length of the chain and to $v_0$ and $v_\ell$ as its endpoints.

\item By an \textbf{odd} (resp. \textbf{even}) chain we mean that the length (i.e., the number of edges on it) is odd (resp. even).

\item By a \textbf{path}, we mean a simple path, i.e. a chain with distinct vertices, and by a \textbf{cycle} we mean a simple cycle, i.e. a chain whose endpoints are equal but all other vertices are distinct.

\item An edge $e_i$ on an odd path $v_0 e_1 v_1 e_2 \dots e_{k+1} v_{2k+1}$ is called \textbf{even} (resp. \textbf{odd}) if $i$ is even (resp. odd). Note that this would not change if the path was written in the reverse order.

\item If $A, B\subseteq E$, then an \textbf{$A - B$ alternating chain} is a chain such that edges of odd index are in $A\setminus B$ and edges of even index are in $B \setminus A$, or vice versa. Note that if $A$ and $B$ are both matchings, then the chain must be either a path or a cycle, or else some vertex will be incident to two edges of the same matching.

\item If $A\subset E$ is a subset of the edges of the graph $G$, then $\boldsymbol{A^c}$:= $E \setminus A$ is the complement in $E$ of $A$.

\end{itemize}

\noindent The majority of the techniques used in \cref{sec:ratio<1} make use of the properties of the alternating paths of various matchings.

\section{PEC decompositions}\label{sec:pec}
   Now, we present a type of decomposition for graphs that highlights matching structure and gives an alternative way to compute the parameters $\nu(G)$, $\lambda(G)$, and $\mu(G)$. Several of the following lemmas will be used in the proof of our main theorems. Throughout this section, let $G := (V,E)$ be a graph and $n := |V|$.
    \begin{itemize}

	\item A \textbf{subgraph} of $G$ is a graph $G' := (V', E')$ with $V' \subseteq V$ and $E' \subseteq E$. A subgraph $G' := (V',E')$ of $G$ is \emph{spanning} if $V' = V$.
    \item A \textbf{pec decomposition} of $G$ is a spanning subgraph $G' := (V,E')$, each of whose (connected) components is a path or an even cycle. (Isolated vertices are treated as even paths.)
     \item If $A$ is a subset of $E$, then we use the notation $\boldsymbol{G_{A}}:=(V,A)$ to denote the spanning subgraph of $G$ with edge set $A$. In particular, we use this notation in the case that $A$ is the union of certain matchings.
    \end{itemize}
    For a pec decomposition $G'$ of $G$, let 
    
    \begin{itemize}
    \item $\boldsymbol{\p(G')}$ (resp., $\boldsymbol{\ep(G')}$) denote the number of components of $G'$ that are paths (resp., even paths);

    \item $\boldsymbol{\p(G)} := \min\{\p(G') : \text{$G'$ is a pec decomposition of }  G\}$;
	
    \item $\boldsymbol{\ep(G)} := \min\{\ep(G') : \text{$G'$ is a pec decomposition of } G\} $.
    \end{itemize}
These last three parameters describe structural properties of the graph that can be assessed with matchings and alternating paths. They are a direct parallel to our original $\nu, 
\mu$, and $\lambda$. For our first lemma, we show that $\nu(G)$ and $\ep(G)$ are directly related.
\begin{lemma}\label{max_in_pec}
	For any pec decomposition $G'$ of $G$, $\nu(G') = \frac{n - \ep(G')}{2}$.
\end{lemma}
\begin{proof}
	Let $M$ be a maximum matching for $G'$. It is clear that $M$ is a perfect matching on each odd path and even cycle of $G'$ and a near perfect matching for each even path. Thus, $M$ misses exactly one vertex from each even path and no other vertex. Hence, $\ep(G') + 2|M| = n$.
\end{proof}
\begin{prop}\label{nu_via_e}
	$\nu(G) = \frac{n - \ep(G)}{2}$. In fact: 
	\begin{enumerate}
		\item For any maximum matching $M$, $G_M$ is a pec decomposition with $\ep(G_M) = \ep(G)$.
		
		\item Conversely, for any pec decomposition $G'$ with $\ep(G') = \ep(G)$, $\nu(G') = \nu(G)$.
	\end{enumerate}
\end{prop}
\begin{proof}
	For $\nu(G) \le \frac{n - \ep(G)}{2}$, let $M$ be a maximum matching. Then $G_M$ is a pec decomposition and $M$ is the unique maximum matching in $G_M$, so by \cref{max_in_pec}, $|M| = \frac{n - \ep(G_M)}{2} \le \frac{n - \ep(G)}{2}$.
	
	For $\nu(G) \ge \frac{n - \ep(G)}{2}$, let $G'$ be a pec decomposition with $\ep(G') = \ep(G)$, so by \cref{max_in_pec}, $\nu(G) \ge \nu(G') = \frac{n - \ep(G')}{2} = \frac{n - \ep(G)}{2}$.
	
	Thus, we have $\nu(G) = \frac{n - \ep(G)}{2}$, so looking back, we see that the inequalities in the first two paragraphs are equalities, and hence $\ep(G_M) = \ep(G)$ and $\nu(G') = \nu(G)$.
\end{proof}
We can find a similar correspondence between $\lambda(G)$ and $\p(G)$ if we consider the pairs of disjoint matchings in $G$ instead of the maximum matchings.
\begin{lemma}\label{total_in_pec}
	For any pec decomposition $G' := (V, E')$ of $G$, $\lambda(G') = |E'| = n - \p(G')$.
\end{lemma}
\begin{proof}
	Let $G' := (V,E')$ be a pec decomposition and let $(H,H')$ be a pair of disjoint matchings of $G'$ with $|H| + |H'| = \lambda(G')$. It is clear that $H \cup H' = E'$ because each component of $G'$ can be taken as an $(H,H')$-alternating path/cycle i.e. a path/cycle whose edges alternate between being in $H$ and $H'$.
	
	It remains to show that $|E'| = n - \p(G')$. Since $G'$ is pec, we have $|E'| = e_p + e_c$, where $e_p$ is the number of edges in paths of $G'$ and $e_c$ is the number of edges in even cycles of $G'$. Since paths have one more vertex than they have edges and cycles have as many vertices as edges, $e_p = n_p - \p(G')$ and $e_c = n_c$, where $n_p$ (resp., $n_c$) is the number of vertices on paths (resp., cycle) of $G'$. Thus, $|E'| = e_p + e_c = n_p - p(G') + n_c = n-\p(G')$.
\end{proof}

\begin{prop}\label{lambda_via_p}
	$\lambda(G) = n - \p(G)$. In fact: 
	\begin{enumerate}
		\item For any $(H,H') \in \Lambda(G)$, $G_{H \cup H'}$ is a pec decomposition with $\p(G_{H \cup H'}) = \p(G)$.
		
		\item Conversely, for any pec decomposition $G'$ with $\p(G') = \p(G)$, $\Lambda(G') \subseteq \Lambda(G)$.
	\end{enumerate}
\end{prop}
\begin{proof}
	For $\lambda(G) \le n - \p(G)$, let $(H,H') \in \Lambda(G)$. Then every vertex in the spanning subgraph $G_{H \cup H'}$ has degree at most 2 in $G_{H\cup H'}$. Additionally, there can be no odd cycles in $G_{H \cup H'}$ because $H$ and $H'$ are disjoint. $G_{H \cup H'} $ is thus a pec decomposition, so by \cref{total_in_pec}, $\lambda(G) = |H| + |H'| = \lambda(G_{H \cup H'}) = n - \p(G_{H \cup H'}) \le n - \p(G)$.
	
	For $\lambda(G) \ge n - \p(G)$, let $G'$ be a pec decomposition with $\p(G') = \p(G)$. Then by \cref{total_in_pec}, $\lambda(G) \ge \lambda(G') = n - \p(G') = n - \p(G)$.
	
	Thus, we have $\lambda(G) = n - \p(G)$, so looking back, we see that the inequalities in the first two paragraphs are equalities, and hence $\p(G_{H \cup H'}) = \p(G)$ and $\lambda(G') = \lambda(G)$.
\end{proof}
 Now, define 
\[
\boldsymbol{\epp(G)} := \min\set{\ep(G') : \text{$G'$ is a pec decomposition of } G \text{ with } \p(G') = \p(G)},
\]
i.e. $\epp(G)$ is the minimum number of even paths in a pec decomposition that minimizes the number of paths. If we now take $(H,H')\in$ $\Lambda_\mu$ instead of $\Lambda$, we obtain a result for $\mu(G)$ and $\epp(G)$.

\begin{lemma}\label{mu_via_ep_p}
	$\mu(G) = \frac{n - \epp(G)}{2}$. In fact: 
	\begin{enumerate}
		\item For any $(H,H') \in \Lambda_\mu(G)$, $G_{H \cup H'}$ is a pec decomposition with $\p(G_{H \cup H'}) = \p(G)$ and $\ep(G_{H \cup H'}) = \epp(G)$.
		
		\item Conversely, for any pec decomposition $G'$ with $\ep(G') = \ep(G)$, $\Lambda_\mu(G') \subseteq \Lambda_\mu(G)$.
	\end{enumerate}
\end{lemma}
\begin{proof}
	For $\mu(G) \le \frac{n - \epp(G)}{2}$, let $(H,H') \in \Lambda_\mu(G)$. Then $G_{H \cup H'}$ is a pec decomposition. By \cref{total_in_pec,lambda_via_p}, $\lambda(G) = |H| + |H'| = n - \p(G_{H \cup H'})$, so $\p(G_{H \cup H'}) = \p(G)$. Furthermore, $H$ is a maximum matching in $G_{H \cup H'}$, so by \cref{max_in_pec}, $|H| = \frac{n - \ep(G_{H \cup H'})}{2} \le \frac{n - \epp(G)}{2}$.
	
	For $\mu(G) \ge \frac{n - \epp(G)}{2}$, let $G' := (V, E')$ be a pec decomposition with $\p(G') = \p(G)$ and $\ep(G') = \epp(G)$. Let $(H,H') \in \Lambda_\mu(G')$. By \cref{lambda_via_p}, $|H| + |H'| = |E'| = n - \p(G') = n - \p(G) = \lambda(G)$, so $(H,H') \in \Lambda(G)$. Moreover, $H \cup H' = E'$ and $H$ is a maximum matching for $G'$, so by \cref{max_in_pec}, $\mu(G) \ge |H| = \frac{n - \ep(G')}{2} = \frac{n - \epp(G)}{2}$.
	
	Thus, we have $\mu(G) = \frac{n - \epp(G)}{2}$, so looking back, we see that the inequalities in the first two paragraphs are equalities, and hence $G_{H \cup H'}$ achieves $\p(G)$ and $\epp(G)$, and $\Lambda_\mu(G') \subseteq \Lambda_\mu(G)$.
\end{proof}
These new $e, \epp,$ and $p$ parameters provide as much information as our original $\nu, \mu$, and $\lambda$. In fact, we can now restate theorems concerning the ratio $\frac{\mu(G)}{\nu(G)}$ in terms of our new pec parameters.
\begin{cor}
	$\mu(G) = \nu(G)$ if and only if $\ep(G) = \epp(G)$.
\end{cor}
\begin{proof}
	Follows from \cref{nu_via_e,mu_via_ep_p}.
\end{proof}

\begin{example}
	If $G$ is the spanner, $\ep(G) = 0$ and it is achieved by a pec decomposition $G'$ into a path of length $5$ and two paths of length $1$. Therefore, $\nu(G) = \frac{10 - 0}{2} = 5$. On the other hand, $\p(G) = 2$ and it is achieved by a unique pec decomposition $G''$ into two paths of length $4$. Thus, $\lambda(G) = 10 - 2 = 8$. Because such a $G''$ is unique, it also witnesses $\epp(G) = 2$, so $\mu(G) = \frac{10 - 2}{2} = 4$.
\end{example}
We present one more brief result for pec decompositions that is useful in studying size differences of disjoint matchings.
\begin{cor}
	$\mu(G) - \mu'(G) = \p(G) - \epp(G)$. In particular, $\mu'(G) = \frac{n - 2 \p(G) + \epp(G)}{2}$.
\end{cor}
\begin{proof}
	The first assertion follows from \cref{mu_via_ep_p} because for any pec decomposition $G'$ and any $(H,H') \in \Lambda_\mu(G')$, 	$\mu(G) - \mu'(G) = |H| - |H'|$ is equal to the number of components of $G'$ that are odd paths, namely, $\p(G') - \ep(G')$. The second equality is by \cref{mu_via_ep_p}.
\end{proof}

\section{Basics tools}\label{sec:tools}

\noindent In this section we will state and prove several structural lemmas about alternating chains (i.e., paths or cycles) that will be used to prove our main theorems. Let $G \defeq (V,E)$ be a finite connected graph. Among the maximum matchings and maximum disjoint matchings in $G$, there are some that have maximal intersection, and this condition allows for sharper restrictions on the possible edge configurations. More specifically, we define \emph{maximally intersecting} matchings:
\begin{defn}\label{maximal}
Matchings $M$, $H$, and $H'$ are called \emph{maximally intersecting} if they satisfy the following:
\begin{enumerate}[(i)]
	\item \label{conditions:M_maximum} $M$ is a maximum matching.
	
	\item \label{conditions:HH'_Lambda_mu} $(H,H') \in \Lambda_\mu(G)$ with $|H| = \mu(G)$.
	
	\item \label{conditions:M_intersect_HH'} $|M \cap (H \cup H')|$ is maximum among all $M,H,H'$ satisfying \labelcref{conditions:M_maximum}--\labelcref{conditions:HH'_Lambda_mu}.
	
	\item \label{conditions:M_intersect_H} $|M \cap H|$ is maximum among all $M,H,H'$ satisfying \labelcref{conditions:M_maximum}--\labelcref{conditions:M_intersect_HH'}.
\end{enumerate}
\end{defn}
\medskip

As we will demonstrate in the next section, these matchings capture the most important structural behavior of $\frac{\mu(G)}{\nu(G)}<1$ graphs. Throughout this section, let $M$, $H$, $H'$ be maximally intersecting matchings. Here are a few basic facts about $M$-$H$ alternating chains. Define an \emph{end-edge} of a path to be an edge in the path incident to one of the path's terminal vertices.

\begin{lemma}\label{known_facts}\
	\begin{enumerate}[(\ref*{known_facts}.a)]
		\item \label{MH-chains_odd} The maximal $M$-$H$ alternating chains are odd paths with end-edges in $M$.
						
		\item \label{MH-endedges_H'} The end-edges of $M$-$H$ alternating paths are in $H'$.

		\item \label{MH-vertices_incident_H'} All vertices on maximal $M$-$H$ alternating paths are incident to $H'$.
	\end{enumerate}
\end{lemma}
\begin{proof}
    
    \noindent\labelcref{MH-chains_odd}  Let $P$ be such a chain. Since these chains are alternating, $P$ cannot be an odd cycle. If $P$ is an even path or cycle, define $M'\defeq(M\setminus P)\cup(P\setminus M)$. $M'$ is a matching since $M$ is a matching and since $P$ is alternating. Furthermore, $|M'|=|M|$, and so $M'$ is a maximum matching. However, $|M'\cap (H\cup H')|\geq|M\cap (H\cup H')|$ since not all $M\cup P$ edges may have been in $H'$, and since $|M'\cap H|>|M\cap H|$, either of $\labelcref{conditions:M_intersect_HH'}$ or $\labelcref{conditions:M_intersect_H}$ must be violated, a contradiction. Now suppose that some end-edge of $P$ is not in $M$. This edge, $e_0$, is in $H\setminus M$, and is adjacent to exactly one $M$ edge $e_1$ in $P$. But then if $e_1$ is removed from $M$ and $e_0$ is added to $M$, $|M\cap H|$ is again increased, a contradiction.
    
    \medskip
    
    \noindent\labelcref{MH-endedges_H'}
    Consider an end-edge $e_0$ in $M\setminus H$. Again note that $e_0$ can be incident to $H$ on only one side. If $e_0$ is not in $H'$, then we can remove $e_1$, the preceding edge, from $H$, add $e_0$ to $H$, and again increase $|M\cap(H\cup H')|$.
    
    \medskip
    
    \noindent\labelcref{MH-vertices_incident_H'}
    Suppose instead some interior vertex is not incident to $H'$, and note that this vertex is incident to no end-edges. This vertex is incident to exactly one edge in $M$, which we call $e_M$. $e_M$ is incident to $H'$ at most once. Via a similar argument to above, if we take $e_M$ to be in $H'$ instead of this adjacent edge, we increase $|M\cap(H\cup H')|$.
\end{proof}
We now present similar structural lemmas about maximal $H$-$H'$ alternating paths.
		
\begin{lemma}\label{new_facts}\
	\begin{enumerate}[(\ref*{new_facts}.a)]
		\item \label{HH'-odd-paths_begin_H} If a maximal $H$-$H'$ alternating path is odd, then its end-edges are in $H$.
		
		\item \label{MH-endedges_HH'-even} The end-edges of maximal $M$-$H$ alternating paths (which we know are in $H'$) are end-edges of even $H$-$H'$ alternating paths.
		
		\item \label{HH'-end-vertices} All inner (not end) vertices of maximal $M$-$H$ alternating paths are inner vertices of $H$-$H'$ alternating paths. In other words, each end-vertex of an $H$-$H'$ alternating path is either an end-vertex of a maximal $M$-$H$ alternating path or is not on any $M$-$H$ alternating path.
		
		\item \label{HH'_even_MH>MH'} On every even maximal $H$-$H'$ alternating chain, the edges in $H \cap M$ are at least as many as those in $H' \cap M$.
	
		\item \label{HH'-end-H-edges_M} Each end-edge of a maximal $H$-$H'$ alternating path that is in $H$ is also in $M$.
	\end{enumerate}
\end{lemma}
\begin{proof}
	\labelcref{HH'-odd-paths_begin_H}: Otherwise, we would switch $H$ and $H'$ on that path, keeping $|H \cup H'|$ the same, but increasing $H$. 
	
	\medskip
	
	\noindent \labelcref{MH-endedges_HH'-even}: Each such edge is adjacent to $H$ exactly on one side, so it is an end-edge of a maximal $H$-$H'$ alternating path, which hence must be even by \labelcref{HH'-odd-paths_begin_H}.
	
	\medskip
	
	\noindent \labelcref{HH'-end-vertices} This is just because each inner vertex of a maximal $M$-$H$ alternating path is incident to $H$ by definition and is also incident to $H'$ by \labelcref{MH-vertices_incident_H'}.
	
	\medskip
	
	\noindent \labelcref{HH'_even_MH>MH'} Otherwise, swapping $H$ and $H'$ contradicts condition \labelcref{conditions:M_intersect_H}.
	
	\medskip
	
	\noindent \labelcref{HH'-end-H-edges_M} Every edge in $H \setminus M$ is an inner edge on a maximal $M$-$H$ alternating path, so it cannot be an end-edge of a maximal $H$-$H'$ alternating path by \labelcref{HH'-end-vertices}.
\end{proof}


\section{Subclass of graphs with $\frac{\mu(G)}{\nu(G)}<1$}\label{sec:ratio<1}

We are now prepared to state and prove our main theorems. For $d \in \N$, we call a vertex in a graph a \emph{$d$-vertex} if its degree is $d$, and we say that a graph has degree $d$ if the degree of each vertex is at most $d$. An edge incident to a leaf (a $1$-vertex) is called a \emph{leaf-edge}.

A path in a graph whose end-vertices are leaves is called \emph{leaf-to-leaf}. Call an edge $e$ on an odd path $p$ \emph{odd} (resp. \emph{even}) 
 if it is in an odd (resp. \emph{even}) position in both obvious orderings of the $p$ edges, i.e. the end edges of $p$ are both odd. Similarly, call a vertex $v$ on an even path $p$ \emph{odd} (resp. \emph{even}) if its edge-distance from either end-edge of $p$ is odd (resp. even).

Our main theorems, \cref{theorem:skelehask} and \cref{theorem:khasskele}, establish that the core structural properties of graphs with $\frac{\mu(G)}{\nu(G)}<1$ are due to a graph's $\textit{skeleton}$ structure.


\begin{defn}\label{defn:skeleton}
	A connected nonempty graph $G$ is called a \emph{\skeleton} if it is of degree $3$ and admits a (not necessarily spanning) subgraph $G' \defeq (V',E')$ satisfying the following conditions:
	\begin{enumerate}[(\ref*{defn:skeleton}.i)]
		\item \label{defn:anush-didnt-label-this-lol} $G'$ consists of vertex-disjoint odd leaf-to-leaf paths of length at least $5$.
		
		\item \label{defn:skeleton:3-vertices} For each odd edge on a maximal path in $G'$, either both of its incident vertices have $G$-degree $3$ (in which case, we call this odd edge \emph{rich}) or neither of them do.
		
		\item \label{defn:so-sam-labels-it-instead} Each maximal path in $G'$ contains a $3$-vertex.
		
		\item \label{defn:skeleton:G-V'_degree_2} The vertices in $V \setminus V'$ have $G$-degree at most $2$.
		
		\item \label{defn:scream-loudly-if-you-find-this} The graph $G \setminus V'$ obtained from $G$ by removing $V'$ consists of vertex-disjoint odd paths.
		
		\item \label{defn:skeleton:end-vertices-G'} Let $G''$ the subgraph of $G$ obtained by removing the rich edges. The $G''$-connected component of each end-vertex of a maximal path in $G'$ is an even path and these are the only connected components that are even paths.
		
		\item \label{defn:skeleton:G''_bipartite} $G''$ is bipartite.
		
		\item\label{defn:nobadcycles} Let $M$ be the matching containing each odd edge of $G'$ and each even edge of $G\backslash V'$. Then $G$ contains no $M - M^c$ alternating cycles.
	\end{enumerate}
	\begin{figure}
    \centering
    \includegraphics[width=0.75\textwidth]{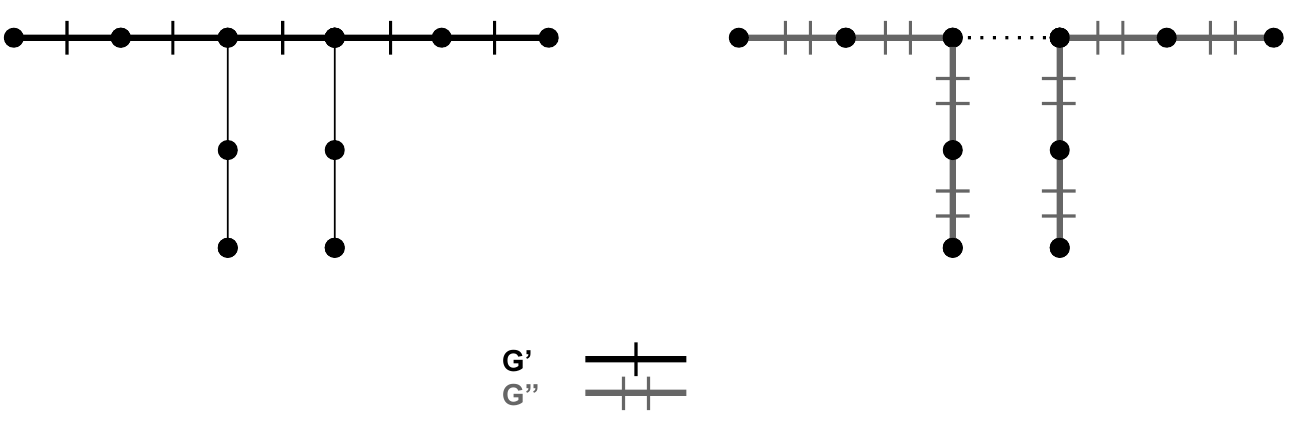}
     \caption{The spanner is a $1$-skeleton}
    \label{fig:1skeleton}
\end{figure}
	Take note in particular of the final condition involving $M - M^c$ alternating cycles. Some of the following proofs are argued by contradiction to construct such a cycle. Call a skeleton a \emph{$k$-skeleton} if $G'$ has $k$-many connected components (odd paths). The spanner is an example of a $1$-skeleton.
\end{defn}

\begin{remark}
    It follows from \labelcref{defn:skeleton:3-vertices,defn:skeleton:G-V'_degree_2} that removing the rich edges from $G$ results in a graph $G''$ of degree $2$, so each connected component of $G''$ is a path or a cycle. Condition \labelcref{defn:skeleton:end-vertices-G'} demands that the connected components of the end-vertices of the maximal paths in $G'$ are exactly those that are even paths (necessarily of length at least $4$). Condition \labelcref{defn:skeleton:G''_bipartite} demands that there are no odd cycles in $G''$.
\end{remark}
The main theorems of this paper, \cref{theorem:skelehask} and \cref{theorem:khasskele}, make the following characterization. A connected graph $G=(V,E)$ is a $k$-skeleton if and only iff $\nu(G)-\mu(G)=k$ and any $M,H,H'$ satisfying conditions \labelcref{conditions:M_maximum}--\labelcref{conditions:M_intersect_H} cover all of $E$, i.e., $E = M \cup H \cup H'$. This condition for an edge covering of $G$ is not arbitrary; the key structural properties in a graph that cause $\frac{\mu(G)}{\nu(G)}<1$ are included in the definition of skeleton, and there are very few ways that edges not in this choice of $(M,H,H')$ can be added to a skeleton $G$ without forcing $\frac{\mu(G)}{\nu(G)}=1$. A brief discussion for these remaning unmatched edges appears in the last section.

We prove a proposition before proving our main theorems.
\begin{prop}\label{prop:uniquem}
	Every skeleton admits a unique perfect matching.
\end{prop}
\begin{proof}
	Let $G$ be a skeleton graph and let $G'$ be as in \cref{defn:skeleton}. Because the connected components of $G'$ and $G\backslash V'$ are odd paths and each odd path admits a perfect matching, the whole graph $G$ admits a perfect matching (the union of the matchings on each odd path). Call this matching $M$.
	

    It remains to show uniqueness. Suppose towards a contradiction that $G$ contains another perfect matching $M'$. Observe that if $M=M'$ on each of the paths in $G'$ then this forces $M=M'$ on all of $G$. Thus, some path of $G'$ must contain an $M\setminus M'$ edge. Denote some direction along this path as the left, and take the leftmost $M\setminus M'$ edge $e_0$ on this path. We mark the following edges in \cref{fig:badcycle}. Observe that $e_0$ must be rich since $M'$ must contain all $M$ edges left of $e_0$ and $M'$ is perfect. Since $M'$ must saturate the left vertex of $e_0$, it must include another edge incident to this vertex. Observe that there exists an $M^c-M$ alternating path connecting this vertex to another $M\setminus M'$ rich edge elsewhere in $G'$, since if this path terminates in a leaf, $e_0$ would necessarily be in $M'$. Consider the $M\setminus M'$ edge on the other end of this path and call it $e_1$. The vertex of $e_1$ not incident to this path must also be incident to some non-$M$ edge. One of these edges must be in $M'$ since $M'$ is perfect. If it is a $G'^c$ edge, travel along the $M^c-M$ alternating path it is a part of, and if it is a $G'$ edge, travel along this $G'$ edge to the adjacent $M$ edge. Again, we end on an $M\setminus M'$ edge $e_2$. Repeat this process for $e_2$. This process can never end by travelling into some leaf of $G$, since this would imply that one of the $e_i$ is in $M'$. This process can only terminate if some $e_i=e_0$ for some $i>0$. But observe that the cycle we have created has alternating edges in $M$ and the remaining in $M^c$, contradicting \labelcref{defn:nobadcycles}.
\end{proof}
	\begin{figure}[h]
    \centering
    \includegraphics[width=0.75\textwidth]{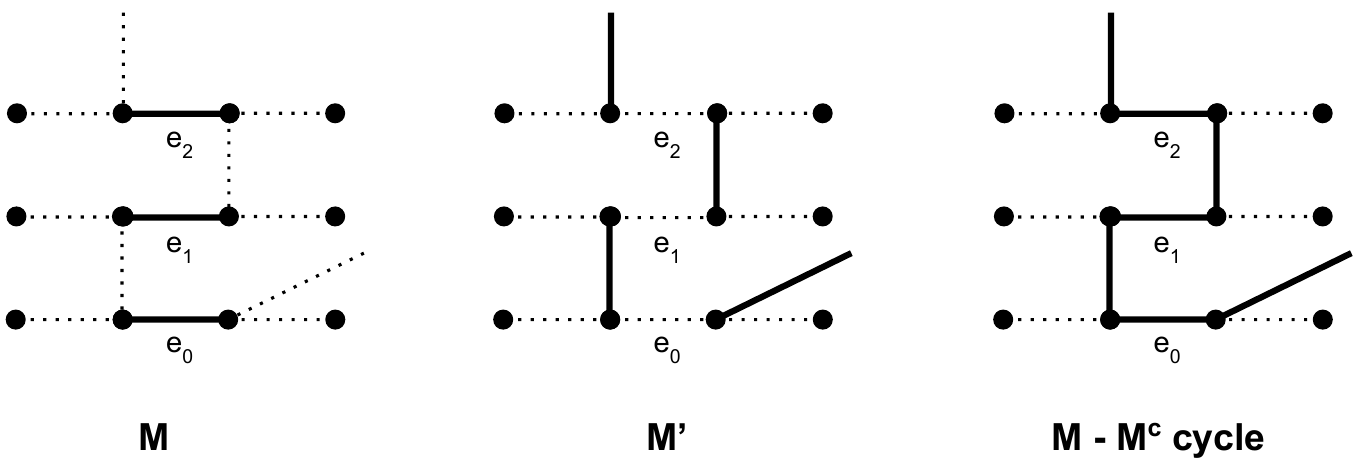}
     \caption{$M$ and $M'$ as in the proof of \cref{prop:uniquem}. $M - H$ paths are not drawn full length.}
    \label{fig:badcycle}
\end{figure}

This unique perfect matching is used in our first main theorem. Our proof uses structural properties derived from the pec parameters of Section 3.
For a pair $(H,H')$ of disjoint matchings, call a vertex $v \in V$ \emph{$(H,H')$-saturated} (or just \emph{saturated} if $(H,H')$ is clear from the context) if $\deg_{G_{H \cup H'}}(v) = \min \set{2, \deg_G(v)}$, where $G_{H \cup H'} \defeq (V, H \cup H')$.

\begin{theorem}\label{theorem:skelehask}
    For a $k$-skeleton $G$, $\nu(G) - \mu(G) = k$.
\end{theorem}
\begin{proof}
    	Let $G'$ be as in \cref{defn:skeleton}. 
	\begin{itemize}
		\item Let $M$ be the unique perfect matching of $G$. 
		
		\item Put in $H$ all the even edges of the maximal odd paths of $G'$ and all the odd edges of maximal odd paths of $G \setminus V'$. 
		
		\item Put in $H'$ all the non-rich odd edges of $G'$, all even edges of the maximal odd paths of $G \setminus V'$, and all edges of $G$ that are not in $G'$ and $G \setminus V'$.
	\end{itemize}
		\begin{figure}[h]
    \centering
    \includegraphics[width=0.75\textwidth]{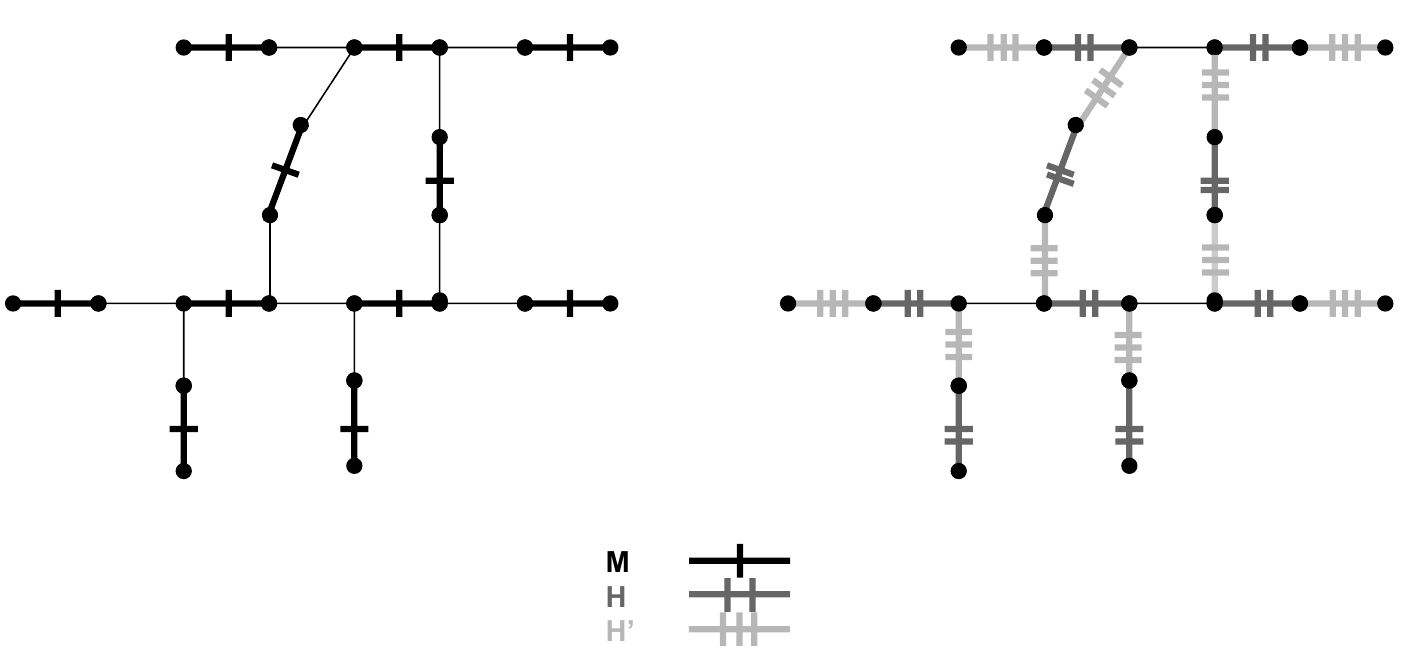}
     \caption{$M$, $H$, and $H'$ as in the proof of \cref{theorem:skelehask}. Depicted is a 2-skeleton. }
    \label{fig:matchings}
\end{figure}
	 These matchings are illustrated in \cref{fig:matchings} on a 2-skeleton. Notice that $M\cup H\cup H'$ covers the edges of $G$ and that $H$ and $H'$ are disjoint. Observe also that all of the $M \setminus H$ and $H \setminus M$ edges must occur in $G'$. Then $|M|-|H|=k$, as there is one more $M$ edge than $H$ edge along each path of $G'$. If we can show that $|H|=\mu(G)$, then we have $\nu(G)-\mu(G)=|M|-|H|=k$. Note that every non-leaf vertex of $G$ is saturated by both $H$ and $H'$, and every leaf is saturated by either $H$ or $H'$. Hence, all vertices are saturated, so $(H,H') \in \Lambda(G)$, and so by \cref{lambda_via_p}, the corresponding pec decomposition achieves $\p(G)$. We will now show that the pec decomposition formed by $H\cup H'$ is the unique pec decomposition achieving $\p(G)$.
	
	
     Suppose that there exists some other pec decomposition $E'\subseteq E$ that achieves $\p(G)$. Every leaf in $G$ must be a leaf of $E'$, and since $E'$ achieves $\p(G)$, these are the only leaves of $E'$. $E'$ must have the same size as $H\cup H'$, and so there must exist an $H\cup H'$ edge not in $E'$. Call this edge $e_0$. We mark the following edges in \cref{fig:pecs}. Each vertex of $e_0$ must be degree 3, since otherwise, the vertex would be a leaf in $E'$. Since $M,H,H'$ is a cover, this vertex must be adjacent to a rich $M$ edge in $E'$. The opposite vertex of this $M$ edge must also be degree three. One of the adjacent edges in $H$ or $H'$ must be in $E'$ to prevent this vertex from being a leaf in $E'$. Call the edge not in $E'$ $e_1$. The other vertex of $e_1$ must be degree 3 to avoid making new leaves. The M rich edge adjacent to this vertex must be adjacent to another two $H$ and $H'$ edges. Of these two, the one not in $E'$ is $e_2$. We may repeat this argument to deduce further $e_i$. This process of following edges can only terminate if some $e_i=e_0$. But then we have followed an $M- M^c$ even alternating cycle, which is forbidden by \labelcref{defn:nobadcycles}. Thus, the only pec decomposition of $G$ is the one formed by $H\cup H'$. 
		\begin{figure}[h]
    \centering
    \includegraphics[width=0.75\textwidth]{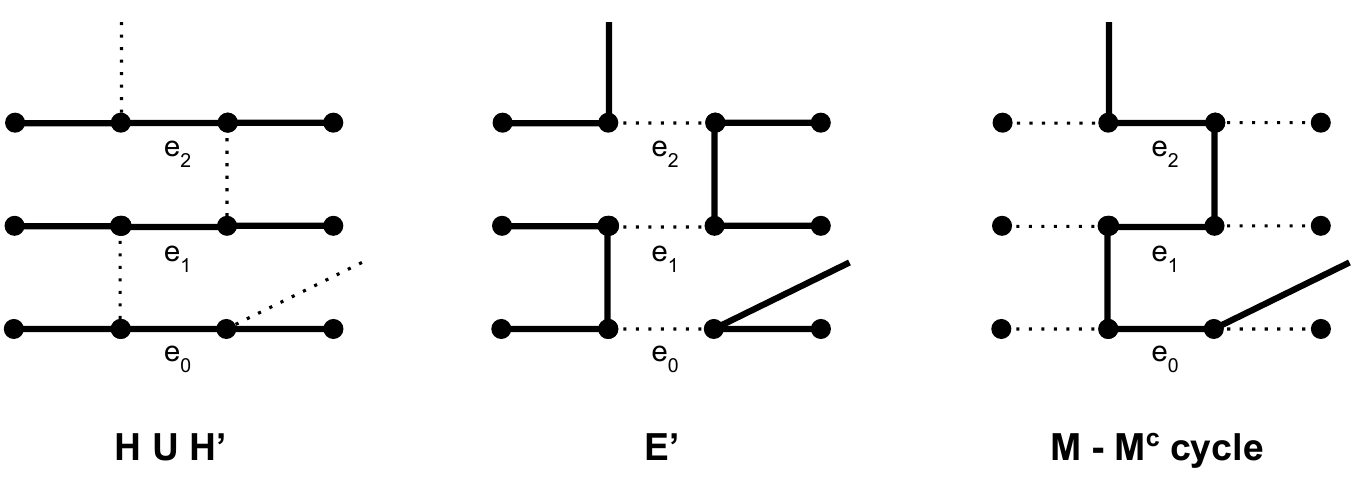}
     \caption{$H\cup H'$ and $E'$ as in the proof of \cref{theorem:skelehask}. $H\cup H'$ paths are not drawn full length.}
    \label{fig:pecs}
\end{figure}

	To complete the proof, we will show that the given $(H,H')$ maximizes the size of $H$ among the pairs in $\Lambda$. Every $(H,H')\in \Lambda$ shares the same pec decomposition, so we need only show that the end-edges of each maximal odd $H - H'$ alternating path are in $H$. Consider some $H'$ edge that is the end-edge of a maximal $H - H'$ alternating path. By definition, this must be an end-edge of a path in $G'$. But then the resulting $H - H'$ alternating path is even length by \labelcref{defn:skeleton:end-vertices-G'}.
\end{proof}

		
		
	
For our second main theorem, we now prove that the converse is true if $G$ is covered by a certain choice of $M,H,H'$.
\begin{theorem}\label{theorem:khasskele}
	Let $G$ be a connected graph with $\dfrac{\mu(G)}{\nu(G)} < 1$ and suppose that any (equivalently, some) maximally intersecting $M,H,H'$ cover all of $E$, i.e., $E = M \cup H \cup H'$, then $G$ is a skeleton. In particular, $G$ is a $(\nu(G)-\mu(G))$-skeleton.
\end{theorem}
\begin{proof}
	Since each vertex is incident to each of $M, H,$ and $H'$, and since $E=M\cup H\cup H'$, $G$ is a graph of degree 3. We will show that taking $G'$ to be the set of $M\cup H$ edges in a nontrivial $M-H$ alternating path satisfies conditions \labelcref{defn:anush-didnt-label-this-lol}-\labelcref{defn:nobadcycles}.
\medskip
	
\noindent\labelcref{defn:anush-didnt-label-this-lol}: Paths being odd length follows from \labelcref{MH-chains_odd}. If the length of some path were exactly three, then this path cannot be connected to the rest of the graph since the end vertices cannot be incident to $H$ and $E=M\cup H \cup H'$.

\medskip

\noindent\labelcref{defn:skeleton:3-vertices}: The internal $M$ edges of each path must either be in $H'$ or have vertices adjacent to $H'$ by \labelcref{MH-vertices_incident_H'}.

\medskip

\noindent\labelcref{defn:so-sam-labels-it-instead}:
If instead some path does not contain a 3-vertex, then every $M$ edge on this path is an $M\cap H'$ edge. Swapping $H$ and $H'$ along this path will keep the same $|M\cap(H\cup H')|$ but increase $|M\cap H|$, a contradiction.

\medskip

\noindent\labelcref{defn:skeleton:G-V'_degree_2}:
The vertices in $V\setminus V'$ certainly have $G-$degree at most $2$. But if some $V\setminus V'$ edge is incident to $M$, this $M$ edge must also be in $H$, as otherwise, the $M$ edge is part of some maximal $M-H$ alternating path in $G'$.

\medskip

\noindent\labelcref{defn:scream-loudly-if-you-find-this}:
The only $G\setminus V'$ edges that can be incident to $G'$ are $H'$ edges. Observe that any $M$ edge not in $G\setminus V'$ must be an $M\cap H$ edge or else it would be part of a maximal $M-H$ alternating path of length at least 3 and would be in $G'$. $G\setminus V'$ must then be degree 2. Observe also that any $H'$ edge incident to $G'$ on some side must be incident to $M$ on the other, or else the $M$ edge adjacent in $G'$ could be moved onto the $H'$ edge to increase $M\cap(H\cup H')$. Similarly, for any $H'$ edge not in $G'$, there must exist a $H'-(M\cap H)$ alternating path connecting it to $G'$. Then this $H'$ edge must be incident to $M$ on its other side, or else we could shift $M$ down this path, including the $M$ edge in $G'$, and increase intersection. $G\setminus V'$ is acyclic since there must be some path connecting these vertices to $G'$ and $G\setminus G'$ is degree 2. The connected components of $G\setminus V'$ must then be odd $M-H'$ alternating paths beginning and ending in $M$.
\medskip

\noindent\labelcref{defn:skeleton:end-vertices-G'}:
The rich edges with this $G'$ are exactly the $M\setminus (H'\cup H)$ edges. Consider the maximal $H-H'$ alternating path going through any such end-vertex. This path is even by \labelcref{MH-endedges_HH'-even}. All edges in this path are in $G''$. The connected component of $G''$ containing this path is the path itself since the internal vertices of this path cannot be incident to any edges in $G\setminus G''$ except for at most a single $M\setminus (H\cup H')$ edge. To show these are the only even paths in $G''$, observe that every edge in $G''$ must be in either $H$ or $H'$ since $E=M\cup H\cup H'$ and $G''$ contains no $M\setminus (H\cup H')$ edges. Then in fact every connected component of $G''$ is an $H-H'$ alternating path or cycle. Even length path implies some end is in $H'$. But these can only be the endpoints of the $G'$ paths that are $M\cap H'$ edges, since every other $H'$ edge must be incident to $H$ on both sides.
\medskip

\noindent\labelcref{defn:skeleton:G''_bipartite}:
By above, every connected component must be an $H-H'$ alternating path or cycle. $G''$ can thus have no odd cycles, so it is bipartite.
\medskip

\noindent\labelcref{defn:nobadcycles}:
Suppose towards a contradiction that there was such an $M-M^c$ alternating cycle $C$. Note that $M^c=(H\cup H')$. But if we moved the $M$ edges of $C$ onto the $M^c$ edges of $C$ we would increase $|M\cap(H\cup H')|$ since $C$ must contain some $M\setminus (H\cup H')$ edge.\medskip

\noindent Note that, following \labelcref{MH-chains_odd}, every connected component of $G'$ is an odd $M-H$ alternating path with end edges in $M$, and that by construction, every $(M\setminus H)\cup(H\setminus M)$ edge appears in $G'$. Therefore, $k=\#$(paths in $G'$)$ =\nu(G)-\mu(G)$, since there is exactly one more $M$ edge than $H$ edge along each of these paths.
\end{proof}

\section{Possible edges in $E \setminus (M \cup H \cup H')$}\label{sec:extra-edges}

Characterizing the graphs with ratio less than 1 and covered by our special choice of matchings was possible due to the restrictions on the configurations in which $M$, $H$, and $H'$ edges can appear along alternating chains. The edges not appearing in these maximally intersecting matchings do not contribute to the overall structure of graphs with $\frac{\mu(G)}{\nu(G)}<1$. In this section, we demonstrate that these remaining edges can only appear in specific configurations. There are many questions still open for the edges not in any of the maximally intersecting matchings. Fully understanding the ways in which these unmatched edges can appear in a graph would lead to a full characterization of graphs with $\frac{\mu(G)}{\nu(G)}<1$. 

\smallskip

Throughout, let $G$ be an arbitrary finite connected graph.

\begin{lemma}\label{unsaturated_vertices_not_adjacent} 
	For any $(H,H') \in \Lambda_\mu(G)$, no two $(H,H')$-unsaturated vertices are adjacent (in $G$), unless they are the end-vertices of the same even maximal $H$-$H'$ alternating path.
\end{lemma}
\begin{proof}
	Let $u,v \in V$ be unsaturated and assume $(u,v) \in E$. If both $u$ and $v$ are not incident to $H$ (resp. $H'$), then adding $(u,v)$ to $H'$ (resp. $H$) increases $H \cup H'$, a contradiction. Thus, say $u$ is incident to $H$ but not $H'$ and $v$ is incident to $H'$ but not $H$. Then both $u$ and $v$ are end-vertices of maximal $H$-$H'$ alternating paths $P_u$ and $P_v$. If these paths are different, then we can switch the $H$ and $H'$ on $P_v$, making $v$ incident to $H$ and reducing to the previous case, i.e. adding $(u,v)$ to $H$ and increasing $H \cup H'$.
\end{proof}

\begin{lemma}\label{even_end-vertex_not_adjacent_odd_path_or_odd vertex} 
	For any $(H,H') \in \Lambda_\mu(G)$, no end-vertex $u$ of an even maximal $H$-$H'$ alternating path is adjacent to any inner vertex of an odd maximal $H$-$H'$ alternating path or any inner odd vertex of an even maximal $H$-$H'$ alternating path.
\end{lemma}
\begin{proof}
	Let $p$ be an even maximal $H$-$H'$ alternating path, $u$ be an end-vertex of $p$, and $v$ be an inner vertex of an odd maximal $H$-$H'$ alternating path or an inner odd vertex of an even maximal $H$-$H'$ alternating path. Without loss of generality, by swapping $H$ and $H'$ on $p$, we may assume that $u$ is incident to $H'$.
	
	If $u$ and $v$ were adjacent, we could remove the $H$ edge on $q$ incident to $v$ and instead add $(u,v)$ to $H$, thus replacing both $p$ and $q$ with new chains $p'$ and $q'$. Note that $q'$ is an odd path. If $p=q$, then $p'$ is an even cycle; otherwise $p'$ is an odd path. Either way, the number of paths in $G_{H \cup H'}$ doesn't increase, while the number of even paths definitely decreases, contradicting that $G_{H \cup H'}$ achieves $\p(G)$ and $\epp(G)$.
\end{proof}

Perhaps the task of characterizing all unsaturated graphs with ratio $1$ in general is difficult and out of reach at the moment, but one can look at subclasses of graphs and try to find a characterization among those. We do not even know whether the following graphs have ratio $1$.

\begin{defn}
A graph is called \textit{tetris} if it is a finite connected subgraph of the grid $\Z^2$ with no leaves or cutvertices\footnote{A \textit{cutvertex} in a graph is a vertex whose removal increases the number of connected components.}.
\end{defn}

The authors of \cite{IGL:2019} have attempted to prove that tetris graphs have ratio $1$, but unsuccessfully. We state it here as a question.

\begin{question}\label{q:tetris}
Is $\frac{\mu(G)}{\nu(G)} = 1$ for all tetris graphs $G$?
\end{question}




\bibliographystyle{ieeetr}
\bibliography{references}\vspace{0.75in}
\end{document}